\documentclass[11pt]{article}
\usepackage{lmodern}
\usepackage{amsmath,amssymb,amsthm}
\usepackage[margin=1in]{geometry}
\usepackage{url}
\usepackage{booktabs}
\usepackage[hidelinks]{hyperref}
\newtheorem{theorem}{Theorem}
\newtheorem{lemma}[theorem]{Lemma}
\newtheorem{corollary}[theorem]{Corollary}
\newtheorem{proposition}[theorem]{Proposition}

\title{A finite cover for coefficient positivity of stretched
Littlewood--Richardson polynomials in the seven-row, size-thirty box}
\author{Maseeh Ghodsi\\Independent Researcher\\\texttt{maseeh.ghodsi@gmail.com}}
\date{9 September 2026}
\hypersetup{
  pdftitle={A finite cover for coefficient positivity of stretched Littlewood--Richardson polynomials in the seven-row, size-thirty box},
  pdfauthor={Maseeh Ghodsi},
  pdfsubject={Computer-assisted coefficient positivity for the seven-row, size-thirty box},
  pdfkeywords={Littlewood-Richardson coefficients, Ehrhart polynomials, finite enumeration}
}
\begin{document}
\maketitle
\begingroup
\renewcommand{\thefootnote}{}
\footnotetext{\copyright\ 2026 Maseeh Ghodsi, for
rights held in original contributions. This paper is licensed under
Creative Commons Attribution 4.0 International (CC BY 4.0):
\url{https://creativecommons.org/licenses/by/4.0/}.}
\endgroup
\begin{abstract}
We give a finite cover argument for nonnegativity of the ordinary monomial
coefficients of every stretched Littlewood--Richardson polynomial with
partition lengths at most seven and outer size at most thirty. Explicit
reductions leave 358,952 residual triples. The computational part consists
of exact finite enumeration, local reduction checks and rational Ehrhart
polynomial computations. Its three named dependencies are stated precisely
below, separately from the mathematical implication they establish.
\end{abstract}
\section{Statement and mathematical inputs}
Partitions are finite weakly decreasing sequences of positive integers,
with zero padding allowed. A triple is balanced if
$|\lambda|=|\mu|+|\nu|$; the outer partition is $\lambda$.
The box in the Epoch problem~\cite{Epoch} consists of balanced triples
with all lengths at most seven and outer size at most thirty. Balance implies the other size bounds.
\textbf{Standing convention.} Every triple in a polynomial identity below
is a balanced triple of partitions; every displayed part beyond a partition's
length is zero. Subtracting parts is permitted only when the displayed
hypotheses give nonnegative weakly decreasing sequences, after which trailing
zeros are removed. A rank is an integer at least the three lengths.

Write $P_T(t)=P^\lambda_{\mu\nu}(t)$ for the unique polynomial over
$\mathbb Q$ satisfying
\[
 \forall t\in\mathbb Z_{\ge1},\qquad P_T(t)=c^{t\lambda}_{t\mu,t\nu}.
\]
For every balanced $T=(\lambda,\mu,\nu)$ with
$\max(\ell(\lambda),\ell(\mu),\ell(\nu))\le7$ and $|\lambda|\le30$,
our target is
\[
 \forall i\in\mathbb Z_{\ge0},\qquad [t^i]P_T(t)\ge0.
\]
Here $[t^i]$ always means the coefficient in $1,t,t^2,\ldots$; coefficients
past the degree are zero. The theorem quantifies over every positive
stretch, including stretches whose resulting partitions exceed the box.
The zero polynomial is characterized by its positive-integer values;
when $c^\lambda_{\mu\nu}=0$, its value at zero is $0$, even though the
all-empty triple has LR coefficient $1$.

We use the Littlewood--Richardson tableau and representation rules, and
exchange of the inner partitions. Polynomiality and the bound
$\deg P_T\le \binom{n-1}{2}$ for rank $n\ge2$ follow from
Rassart~\cite[Corollary 4.2]{Rassart}, following Derksen--Weyman.
For ranks zero and one the polynomial is constant.
The following established implications hold for all balanced triples:
\[
\begin{array}{c|ccc}
c^\lambda_{\mu\nu}&0&1&2\\ \hline
P_T(t)&0&1&t+1.
\end{array}
\]
The first is the contrapositive of saturation~\cite{KT}; the second is
Fulton's theorem~\cite[Section 6.1]{KTW}; the third is
Ikenmeyer's theorem~\cite[Theorem 1.1]{Ikenmeyer}.
For positive base multiplicity, $P_T(0)=1$ by the Ehrhart interpretation.
No positivity conjecture or earlier computational sweep is an input.
The tensor argument below also uses the standard complete reducibility
of finite-dimensional rational $GL_n(\mathbb C)$ representations,
duality of tensor products and the tensor--Hom adjunction; see
Milne~\cite[17.14 and 20.35]{Milne} for the complete reducibility input.

For $0\le B\le30$, let $\mathcal B_B$ mean nonnegativity for
all balanced triples in the box with outer size at most $B$.
Only $\mathcal B_0$ is needed: the all-empty triple has polynomial $1$.
Agreement of polynomials at all positive integers implies equality,
since a nonzero polynomial has at most its degree many roots.

\paragraph{Computational convention.}
This is a conventional computer-assisted proof: its finite assertions are
established by the complete exact computation accompanying the article.
The exact arithmetic implementations and their execution are trusted,
as specified below. The mathematical identities and finite coverage
argument are proved here; no completed proof-assistant verification is asserted.

\section{Tableau arrays and shortening}

\begin{lemma}[Array description]\label{lem:arrays}
Suppose \(\mu\subseteq\lambda\) and all partitions have length at most \(n\). At stretch \(t\ge1\), LR tableaux correspond to nonnegative integer arrays \(x_{j,i}\), \(1\le j,i\le n\), satisfying
\begin{align}
\sum_i x_{j,i}&=t(\lambda_j-\mu_j),
&
\sum_j x_{j,i}&=t\nu_i,                                      \label{eq:margins}\\
x_{j,i}&=0\quad(i>j),                                        \label{eq:support}\\
\sum_{h<j}x_{h,i}-\sum_{h\le j}x_{h,i+1}&\ge0
\quad(1\le i<n),                                            \label{eq:lattice}\\
t(\mu_{j-1}-\mu_j)+\sum_{i<b}x_{j-1,i}
-\sum_{i\le b}x_{j,i}&\ge0
\quad(2\le j\le n,\ 1\le b\le n).                            \label{eq:columns}
\end{align}
Define
\[
q_k=\sum_{j\le k}(\mu_j+\nu_j-\lambda_j).
\]
The number of labels at most \(k\) below the first \(k\) rows is \(tq_k\). Existence of a tableau therefore implies \(q_k\ge0\). It also implies
\[
\lambda_j-\mu_j\le\nu_1.
                                                               \tag{5}
\]
\end{lemma}

\begin{proof}
The array records each row's label counts; weak increase determines the filling uniquely.

For labels \(i,i+1\), their smallest prefix-count difference within row \(j\) occurs after its \(i+1\)'s and before its \(i\)'s are read. This is the left side of \eqref{eq:lattice}. These inequalities are therefore equivalent to the lattice condition.

The first row contains only ones. Inductively, if earlier rows contain only labels at most \(j-1\), an entry \(i>j\) in row \(j\) would be read before any \(i-1\) in that row, and earlier rows contain no \(i-1\). This violates the lattice condition and proves \eqref{eq:support}.

The entries at most \(b\) in row \(j\) end at column
\[
t\mu_j+\sum_{i\le b}x_{j,i}.
\]
The entries less than \(b\) in row \(j-1\), together with its removed inner cells, end at
\[
t\mu_{j-1}+\sum_{i<b}x_{j-1,i}.
\]
Column strictness is equivalent to the first endpoint not exceeding the second for each \(b\). This gives \eqref{eq:columns}.

All cells in the first \(k\) rows have labels at most \(k\). Subtracting their number from the total content of such labels gives \(tq_k\).

For \(n>1\), summing \eqref{eq:lattice} over \(i\) gives
\[
\sum_{i=2}^n x_{j,i}
\le \sum_{h<j}x_{h,1}-\sum_{h<j}x_{h,n}.
\]
Adding \(x_{j,1}\) bounds the row length by the total number \(t\nu_1\) of ones. For \(n=1\), the content equation gives the same bound.
\end{proof}

\begin{theorem}[Strengthened shortening]\label{thm:shortening}
Let \(\mu\subseteq\lambda\), with all partitions of length at most \(n\). Fix \(1\le k<n\), suppose \(q_k\ge0\), and put
\[
r_{k+1}=\lambda_{k+1}-\mu_{k+1},\qquad
b_k=\min(q_k,r_{k+1}).
\]
If an integer \(a\ge0\) satisfies
\[
a\le\lambda_k-\lambda_{k+1},\qquad
a\le\mu_k-\mu_{k+1}-b_k,
                                                               \tag{6}
\]
define
\[
\lambda^-_j=\lambda_j-a\,1_{\{j\le k\}},\qquad
\mu^-_j=\mu_j-a\,1_{\{j\le k\}}.
\]
These are partitions after zero parts are dropped, and
\[
P^\lambda_{\mu\nu}=P^{\lambda^-}_{\mu^-,\nu}.
                                                               \tag{7}
\]
The analogous exchanged-inner identity holds.
\end{theorem}

\begin{proof}
The bounds ensure nonnegative weakly decreasing modified sequences. Their row-length differences are unchanged.

At stretch \(t\), the margins, support, and lattice inequalities remain identical. Only column inequalities crossing rows \(k,k+1\) change.

For \(b\le k\), the lower partial row sum is bounded by both \(tq_k\) and \(tr_{k+1}\). The modified slack is therefore at least
\[
t(\mu_k-\mu_{k+1}-a-b_k)\ge0.
\]
For \(b\ge k+1\), triangular support makes both partial sums complete row sums. The modified slack is
\[
t(\lambda_k-\lambda_{k+1}-a)\ge0.
\]
Thus every original array is modified-feasible. Conversely, each original crossing slack equals the modified slack plus \(ta\); all other conditions coincide. This proves a bijection at every stretch. Positive base multiplicity was not assumed.
\end{proof}

\section{Deletion and factorization}

\begin{lemma}[Empty-row deletion]\label{lem:emptyrow}
If \(\mu\subseteq\lambda\) and \(\lambda_j=\mu_j\), deleting part \(j\) from both gives
\[
P^\lambda_{\mu\nu}
=
P^{\widehat\lambda}_{\widehat\mu,\nu}.
                                                               \tag{8}
\]
If \(j\le\ell(\lambda)\), the outer size decreases by \(\lambda_j>0\).
\end{lemma}

\begin{proof}
Put \(L=\lambda_j=\mu_j\). At stretch \(t\), row \(j\) is empty. All skew cells above it lie in columns greater than \(tL\), because their inner row lengths are at least \(tL\). All skew cells below it lie in columns at most \(tL\), because their outer row lengths are at most \(tL\).

Deleting the empty row introduces no column comparisons across the cut, since the two groups occupy disjoint column ranges. Other comparisons and the reading word remain unchanged. Inserting the row reverses the construction.
\end{proof}

\begin{theorem}[Initial-sum factorization]\label{thm:factor}
Suppose explicitly that $|\lambda|=|\mu|+|\nu|$,
$\mu\subseteq\lambda$, and $1\le k<n$ for a rank $n$, with $q_k=0$.
Splitting all three partitions after part $k$ gives
\[
P^\lambda_{\mu\nu}
=
P^{\lambda_{\le k}}_{\mu_{\le k},\nu_{\le k}}\,
P^{\lambda_{>k}}_{\mu_{>k},\nu_{>k}}.
                                                               \tag{9}
\]
Both triples are balanced. If \(1\le k<\ell(\lambda)\), both outer sizes are strictly smaller.
\end{theorem}

\begin{proof}
The first \(k\) rows use only labels at most \(k\). Since \(q_k=0\), they exhaust those labels. The lower rows use larger labels. Restriction to the upper rows and subtraction of \(k\) from the lower labels produces tableaux of the displayed triples.

Conversely, concatenate a pair of such tableaux, increasing the lower labels by \(k\). Column strictness across the cut follows from the separation of the label ranges. Lattice comparisons within each range are inherited. For the remaining comparison, the upper rows supply \(t\nu_k\) copies of \(k\), while every lower prefix has at most \(t\nu_{k+1}\) copies of \(k+1\). The inequality follows from \(\nu_k\ge\nu_{k+1}\).

This is a bijection at every stretch. Balance follows from \(q_k=0\) and total balance. A proper cut within the positive outer parts gives two positive, smaller outer sizes.

If both factors have nonnegative monomial coefficients, so does their product, since each product coefficient is a sum of products of nonnegative numbers. Hence a negative coefficient in the product requires a negative coefficient in a factor.
\end{proof}

\begin{lemma}[Full columns and scaling]\label{lem:basic}
If $\ell(\lambda),\ell(\mu),\ell(\nu)\le n$,
$\mu\subseteq\lambda$, and $\lambda_n,\mu_n\ge1$, subtracting $1$
from all $n$ parts of $\lambda$ and $\mu$ preserves $P$.
The exchanged-inner version holds as well. If all parts are divisible by \(g\ge1\), and \(Q\) is the divided triple's polynomial, then
\[
P(t)=Q(gt),\qquad [t^i]P=g^i[t^i]Q.
                                                               \tag{10}
\]
\end{lemma}

\begin{proof}
Full-column removal translates every skew cell one column left. At stretch \(t\), translate by \(t\) columns. The reading word and relative column comparisons are unchanged.

For division, the original triple stretched by \(t\) is the divided triple stretched by \(gt\). This proves the polynomial identity. Each \(g^i\) is positive, so coefficient signs are preserved.
\end{proof}

\section{Six representatives}

\begin{lemma}[Rectangular complement]\label{lem:complement}
For \(\ell(\alpha)\le n\) and an integer \(M\ge0\) with \(M\ge\alpha_1\), put
\[
\alpha^{c,M}=(M-\alpha_n,\ldots,M-\alpha_1).
\]
Then
\[
s_{\alpha^{c,M}}(x_1,\ldots,x_n)
=(x_1\cdots x_n)^M s_\alpha(x_1^{-1},\ldots,x_n^{-1}),
                                                               \tag{11}
\]
and consequently
\[
V_\alpha^*\cong\det^{-M}\otimes V_{\alpha^{c,M}}.
                                                               \tag{12}
\]
\end{lemma}

\begin{proof}
Regard a tableau as \(M\) columns, allowing empty columns. Complement each column's entry set in \(\{1,\ldots,n\}\) and reverse the column order.

For adjacent original columns with sets \(A,B\), row weak increase is equivalent to
\[
|A\cap\{1,\ldots,r\}|\ge |B\cap\{1,\ldots,r\}|
\quad(1\le r\le n).
\]
Complementation reverses these inequalities, so the reversed complemented columns satisfy row weak increase. The shape becomes \(\alpha^{c,M}\), and the operation is an involution.

If label \(i\) occurred \(e_i\) times, it now occurs \(M-e_i\) times. Summing weights proves (11). Taking dual characters and using the supplied irreducible Schur-character interpretation gives (12).
\end{proof}

\begin{theorem}[Six polynomial-preserving representatives]\label{thm:six}
Suppose \(c^\lambda_{\mu\nu}>0\). Choose an integer \(n\ge1\)
at least as large as all three partition lengths, pad to \(n\), and assume
\(\mu_n=\nu_n=0\). The all-empty triple separately has polynomial \(1\).
Write
\[
N=|\lambda|,\ m=|\mu|,\ v=|\nu|,\quad
L=\lambda_1,\ h=\lambda_n,\ u=\mu_1,\ w=\nu_1.
\]
Six balanced partition triples of length at most \(n\) have polynomial \(P^\lambda_{\mu\nu}\) and outer sizes
\[
N,\quad nL-m,\quad nL-v,\quad n(u+w)-N,\quad
n(u-h)+v,\quad n(w-h)+m.
                                                               \tag{13}
\]
\end{theorem}

\begin{proof}
For \(K\ge W_1\), define
\[
O_K(W)=(K-W_n,\ldots,K-W_1).
\]
Set \(A=\mu\), \(B=\nu\), \(C=\lambda^{c,L}\). Their last parts are zero and their sizes sum to \(nL\).

For \(G=GL_n\), the representation interpretation and tensor-Hom identification give
\[
c^\lambda_{\mu\nu}
=
\dim\operatorname{Hom}_G
\bigl(\det^L,V_A\otimes V_B\otimes V_C\bigr).
                                                               \tag{14}
\]
Multiplicity equals the indicated Hom dimension because equivariant maps between distinct irreducibles vanish and an equivariant endomorphism of an irreducible is scalar. For the latter statement, an eigenspace of an equivariant endomorphism is a nonzero invariant subspace and hence the whole irreducible.

Expression (14) is symmetric in \(A,B,C\). Their first parts are at most \(L\): containment gives this for \(A,B\), and \(C_1=L-h\). Choosing one as \(W\), outer partition \(O_L(W)\), and the other two as inner partitions gives the first three representatives.

For \(W_n=0\), put
\[
W^*=(W_1-W_n,\ldots,W_1-W_1).
\]
Dualizing each irreducible summand preserves the corresponding dual multiplicity. Formula (12) shows that (14) also equals
\[
\dim\operatorname{Hom}_G
\bigl(\det^{K^*},
V_{A^*}\otimes V_{B^*}\otimes V_{C^*}\bigr),
\qquad K^*=u+w-h.
                                                               \tag{15}
\]
Lemma~\ref{lem:arrays} gives \(L\le u+w\), \(h\le w\), and, by exchange symmetry, \(h\le u\). Thus
\[
K^*-A_1=w-h\ge0,\quad
K^*-B_1=u-h\ge0,\quad
K^*-C_1=u+w-L\ge0.
\]
The three further outer partitions \(O_{K^*}(W)\), with \(W\) chosen from \(A^*,B^*,C^*\), are therefore valid.

All constructions commute with stretching. Thus their polynomial identities follow from the multiplicity identities.

Finally,
\[
|C|=nL-N,\quad |A^*|=nu-m,\quad |B^*|=nw-v,\quad |C^*|=N-nh.
\]
Taking the pair sums in the two triples gives (13).
\end{proof}

\section{The minimum-size domain}

\begin{theorem}[Necessary conditions]\label{thm:domain}
Assume \(\mathcal B_B\), and choose a counterexample of minimum outer size \(N\), if one exists. Put \(n=\ell(\lambda)\). Then:
\begin{enumerate}
\item \(B<N\le30\), the base multiplicity is positive, and the gcd of all positive parts is one.
\item \(\mu_n=\nu_n=0\), and \(q_k>0\) for \(1\le k<n\).
\item \(\lambda_j>\mu_j,\nu_j\) for every \(j\le n\).
\item \(\mu_k,\nu_k\le\lambda_{k+1}\) for \(k<n\).
\item For every $1\le k<n$ with $\lambda_k>\lambda_{k+1}$,
\[
\begin{split}
\mu_k-\mu_{k+1}&\le\min(q_k,\lambda_{k+1}-\mu_{k+1}),\\
\nu_k-\nu_{k+1}&\le\min(q_k,\lambda_{k+1}-\nu_{k+1}).
\end{split}                                                    \tag{16}
\]
\item All six sizes in (13) are at least \(N\), equivalently
\[
\max(m,v)\le nL-N,\quad 2N\le n(u+w),\quad
m\le n(u-h),\quad v\le n(w-h).
                                                               \tag{17}
\]
\end{enumerate}
\end{theorem}

\begin{proof}
The supplied zero-multiplicity implication gives positive base multiplicity. The hypothesis \(\mathcal B_B\) gives \(N>B\).

Common division or an applicable full-column removal would preserve a negative coefficient and decrease the outer size. This proves the gcd assertion and the zero last parts.

The deficits are nonnegative by Lemma~\ref{lem:arrays}. A zero proper deficit permits factorization by Theorem~\ref{thm:factor}; some smaller factor would have a negative coefficient. Hence the proper deficits are positive.

Equality \(\lambda_j=\mu_j\) permits deletion of a positive outer part by Lemma~\ref{lem:emptyrow}; equality with \(\nu_j\) is treated by exchange symmetry. Thus both containments are strict.

Failure of (16) at a positive outer gap allows \(a=1\) in Theorem~\ref{thm:shortening}, giving a smaller counterexample. At such a gap, (16) implies
\[
\mu_k-\mu_{k+1}\le\lambda_{k+1}-\mu_{k+1},
\]
hence \(\mu_k\le\lambda_{k+1}\). At a zero outer gap, containment gives the same conclusion. Exchange symmetry handles \(\nu\).

Finally, a symmetry representative of outer size less than \(N\) would automatically remain in the box: its length is at most \(n\), and each inner size is bounded by its balanced outer size. Minimality therefore makes all six sizes at least \(N\). Rearranging these inequalities gives (17).
\end{proof}

\begin{corollary}[The containing partition and deficit]\label{cor:kappa}
Under Theorem~\ref{thm:domain}, put \(\lambda_{n+1}=0\) and
\[
\kappa_j=\min(\lambda_j-1,\lambda_{j+1}),\qquad
e=\#\{j<n:\lambda_j=\lambda_{j+1}\}.
\]
Then \(\kappa\) is a partition with last part zero,
\[
\mu,\nu\subseteq\kappa,\qquad |\kappa|=N-L-e,
                                                               \tag{18}
\]
and
\[
s=N-2(L+e)\ge0,\qquad
(|\kappa|-|\mu|)+(|\kappa|-|\nu|)=s.
                                                               \tag{19}
\]
If \(s=0\), necessarily \(\mu=\nu=\kappa\). If \(s=1\), up to exchange, one inner partition is \(\kappa\) and the other is obtained by removing one removable corner.
\end{corollary}

\begin{proof}
The coordinatewise minimum of the nonnegative decreasing sequences \(\lambda_j-1\) and \(\lambda_{j+1}\) is nonnegative and decreasing. Its last part is zero.

Strict containment gives \(\mu_j,\nu_j\le\lambda_j-1\); interlacing gives the other bound needed for containment in \(\kappa\). At a positive outer gap, \(\kappa_j=\lambda_{j+1}\). At an equality, \(\kappa_j=\lambda_{j+1}-1\). Summing proves (18).

Subtracting \(|\mu|+|\nu|=N\) from \(2|\kappa|\) gives (19). Both deficits are nonnegative. Deficit zero forces equality of both subpartitions with \(\kappa\); deficit one gives the stated removable-corner description.
\end{proof}

\section{Array notation and bounds}
Use Lemma~\ref{lem:arrays} at stretch one and write
\[
 S_{j,i}=\sum_{h\le j}x_{h,i},\quad S_{0,i}=0,\qquad
 y_{j,b}=\mu_j+\sum_{i\le b}x_{j,i},\quad y_{j,0}=\mu_j.
\]
The lattice and column slacks are respectively
\begin{align}
 L_{j,i}&=S_{j-1,i}-S_{j,i+1}\ge0
       &&(1\le j\le n,\ 1\le i<n),\label{rect:eq:lattice}\\
 C_{j,b}&=y_{j-1,b-1}-y_{j,b}\ge0
       &&(2\le j\le n,\ 1\le b\le n).\label{rect:eq:columns}
\end{align}

\begin{lemma}\label{rect:lem:bounds}
For an array in Lemma~\ref{lem:arrays}, put
\[
R_j(b)=\sum_{i=b}^{n}x_{j,i}.
\]
Then
\begin{equation}\label{rect:eq:suffix}
R_j(b)\le S_{j,b}\le\nu_b.
\end{equation}
Whenever $j>b$,
\begin{equation}\label{rect:eq:diagonal}
y_{j,b}\le\mu_{j-b}.
\end{equation}
Consequently, for $p,q\ge1$,
\begin{equation}\label{rect:eq:weyl}
\lambda_{p+q+1}\le\mu_{p+1}+\nu_{q+1},
\end{equation}
where a part beyond the chosen length is zero.
\end{lemma}

\begin{proof}
For $b<n$, the lattice inequalities give
\[
x_{j,i+1}\le S_{j-1,i}-S_{j-1,i+1}.
\]
Summing over $i=b,\ldots,n-1$ yields
\[
\sum_{i=b+1}^{n}x_{j,i}
 \le S_{j-1,b}-S_{j-1,n}
 \le S_{j-1,b}.
\]
Adding $x_{j,b}$ proves the first inequality in
\eqref{rect:eq:suffix}. For $b=n$, it is simply
$x_{j,n}\le S_{j,n}$. The content condition proves
the second inequality.

Iterating the column inequalities $b$ times gives
\[
y_{j,b}\le y_{j-1,b-1}\le\cdots
 \le y_{j-b,0}=\mu_{j-b},
\]
proving \eqref{rect:eq:diagonal}.

If $p+q+1\le n$, the two established bounds give
\[
\lambda_{p+q+1}
 =y_{p+q+1,q}+R_{p+q+1}(q+1)
 \le\mu_{p+1}+\nu_{q+1}.
\]
Otherwise the left side is zero.
\end{proof}

\section{Rectangular shortening}

For $k\ge1$, write $\alpha-a(1^k)$ for subtraction
of $a$ from the first $k$ parts of $\alpha$.

\begin{theorem}\label{rect:thm:rectangle}
Suppose $\lambda,\mu,\nu$ are balanced partitions with
$\mu,\nu\subseteq\lambda$. Let $p,q\ge1$, put $r=p+q$,
and let $a\ge1$ be an integer satisfying
\begin{equation}\label{rect:eq:hypotheses}
\mu_p-a\ge\mu_{p+1},\qquad
\nu_q-a\ge\nu_{q+1},\qquad
\lambda_r-a\ge
\max\{\lambda_{r+1},\mu_{p+1}+\nu_{q+1}\}.
\end{equation}
Define
\[
\widehat\lambda=\lambda-a(1^r),\qquad
\widehat\mu=\mu-a(1^p),\qquad
\widehat\nu=\nu-a(1^q).
\]
These are balanced partitions, and
\begin{equation}\label{rect:eq:identity}
P^\lambda_{\mu\nu}
 =
P^{\widehat\lambda}_{\widehat\mu\,\widehat\nu}.
\end{equation}
\end{theorem}

\begin{proof}
The inequalities ensure weak decrease and nonnegativity.
The three sizes decrease by $ra$, $pa$, and $qa$;
hence balance is preserved.

Containment of $\widehat\mu$ in $\widehat\lambda$
is unchanged in rows at most $p$ and beyond $r$.
For $p<j\le r$,
\[
\lambda_j-a\ge\lambda_r-a
 \ge\mu_{p+1}+\nu_{q+1}\ge\mu_j.
\]
The containment argument for $\widehat\nu$ uses
$q<j\le r$ and $\nu_j\le\nu_{q+1}$.

Choose $n\ge r$ containing all parts, and put
\[
M=\mu_{p+1},\qquad N=\nu_{q+1}.
\]
Given an original LR array $x$, define
\begin{equation}\label{rect:eq:map}
\widehat x_{p+i,i}=x_{p+i,i}-a
\quad(1\le i\le q),
\end{equation}
leaving all other entries unchanged.

We first prove nonnegativity. By \eqref{rect:eq:suffix},
\[
y_{r,q}=\lambda_r-R_r(q+1)
 \ge\lambda_r-N\ge M+a.
\]
Every column $M+1,\ldots,M+a$ contains skew boxes
in all rows $p+1,\ldots,r$: its index exceeds the
relevant inner parts and is at most $\lambda_r$.
Its bottom entry is at most $q$.
A strictly increasing column of $q$ positive entries
whose bottom entry is at most $q$ must be
$1,2,\ldots,q$. Thus $x_{p+i,i}\ge a$ for
every $1\le i\le q$.

The new row sums are correct: rows $p+1,\ldots,r$
lose $a$ boxes, while the skew row lengths of rows
at most $p$ and beyond $r$ are unchanged.
Each label $1,\ldots,q$ loses $a$ occurrences, giving
the new content.

We calculate the changes in every inequality.
For label $i<q$, subtraction in label $i$ occurs
in row $p+i$, while subtraction in label $i+1$
occurs in row $p+i+1$. Their effects on
\eqref{rect:eq:lattice} cancel because
\[
p+i<j\quad\Longleftrightarrow\quad p+i+1\le j.
\]
Therefore
\begin{equation}\label{rect:eq:Lchange}
\widehat L_{j,i}=
\begin{cases}
L_{j,i}-a,&i=q,\ j\ge r+1,\\
L_{j,i},&\text{otherwise}.
\end{cases}
\end{equation}

For labels $i>q$ no entry of either label changes.
For column inequalities with $2\le j\le p$, both inner
boundary terms decrease by $a$ and neither row prefix changes.
For $j\ge r+2$, neither boundary term nor row prefix changes.
For the column inequalities, at $j=p+1$ the decrease
of $\mu_p$ cancels the decrease of the lower row
prefix, for every $b\ge1$.
For $p+2\le j\le r$, the modified upper entry has
label $j-p-1$ and the modified lower entry has label
$j-p$. Both enter the respective prefixes exactly
when $b\ge j-p$, so their effects cancel.
At $j=r+1$, only the upper entry of label $q$
is modified. Thus
\begin{equation}\label{rect:eq:Cchange}
\widehat C_{j,b}=
\begin{cases}
C_{j,b}-a,&j=r+1,\ b\ge q+1,\\
C_{j,b},&\text{otherwise}.
\end{cases}
\end{equation}
The exceptional cases in these two formulas are absent
if $r=n$.

Every decreased inequality has sufficient slack.
By \eqref{rect:eq:diagonal},
\[
y_{r,q-1}\le M.
\]
For $q=1$, this is the identity
$y_{r,0}=\mu_{p+1}=M$.
Hence \eqref{rect:eq:suffix} gives
\[
S_{r,q}\ge R_r(q)
 =\lambda_r-y_{r,q-1}\ge\lambda_r-M.
\]
For $j\ge r+1$,
\[
L_{j,q}
 =S_{j-1,q}-S_{j,q+1}
 \ge S_{r,q}-N
 \ge\lambda_r-M-N\ge a.
\]

If $r<n$, fix $b\ge q+1$.
Summing the lattice inequalities in row $r+1$
over $i=q+1,\ldots,b-1$ and adding
$x_{r+1,q+1}$ gives
\[
\sum_{i=q+1}^{b}x_{r+1,i}
 \le S_{r+1,q+1}-S_{r,b}.
\]
For $b=q+1$, the sum of inequalities is empty and
this relation is equality.
Combining it with \eqref{rect:eq:suffix} yields
\[
R_r(b)+\sum_{i=q+1}^{b}x_{r+1,i}
 \le S_{r+1,q+1}\le N.
\]
Also $y_{r+1,q}\le M$ by \eqref{rect:eq:diagonal}.
Consequently
\[
\begin{split}
C_{r+1,b}
 &=\lambda_r-R_r(b)
   -y_{r+1,q}-\sum_{i=q+1}^{b}x_{r+1,i}\\
 &\ge\lambda_r-M-N\ge a.
\end{split}
\]
This proves every required new inequality.

Conversely, start with a valid hatted array and add
$a$ at exactly the positions $(p+i,i)$,
$1\le i\le q$. Nonnegativity is preserved and the
original margins are restored.
The displayed change formulas are linear identities for any pair
of arrays related by the stated additions or subtractions; their
derivation did not assume feasibility. Thus equations
\eqref{rect:eq:Lchange} and \eqref{rect:eq:Cchange}, read backwards, show that every inequality is unchanged
or increased by $a$. Thus addition gives a valid
original array and is inverse to \eqref{rect:eq:map}.

We have a bijection for the original LR coefficients.
Applying it to the uniformly stretched partitions
with subtraction amount $ta$ proves equality at
every integer $t\ge1$. Polynomial uniqueness gives
\eqref{rect:eq:identity}.
\end{proof}

\section{Published reduction inputs}
\begin{theorem}[Essential Horn factorization]\label{thm:horn}
Let $n\ge2$, $|\lambda|=|\mu|+|\nu|$,
$\max(\ell(\lambda),\ell(\mu),\ell(\nu))\le n$, and
$c^\lambda_{\mu\nu}>0$.
For $1\le r<n$ and $r$-subsets $I,J,K$ of $[n]=\{1,\ldots,n\}$,
write $I=\{i_1<\cdots<i_r\}$ and
$\tau(I)=(i_r-r,\ldots,i_1-1)$, dropping zeros; define $\tau(J)$ and
$\tau(K)$ in the same way. Suppose
\[
 c^{\tau(K)}_{\tau(I),\tau(J)}=1,\qquad
 \sum_{k\in K}\lambda_k=\sum_{i\in I}\mu_i+\sum_{j\in J}\nu_j.
\]
For a subset, parts are selected in increasing index order; complements
are taken in $[n]$. Then
\[
 P^\lambda_{\mu\nu}
 =P^{\lambda_K}_{\mu_I,\nu_J}
  P^{\lambda_{K^c}}_{\mu_{I^c},\nu_{J^c}}.
\]
Both factors are balanced.
\end{theorem}
\begin{proof}[Justification of the cited input]
We use King--Tollu--Toumazet~\cite[Theorem 1.4]{KTT09} in the exact
subset formulation recorded by Cho--Jung--Moon~\cite[Definition 1.2
and Theorem 1.3]{CJMfpsac}. Their ordered triple $(\lambda,\mu,\nu)$
is our $(\mu,\nu,\lambda)$, and their $\pi$ is our $\tau$.
Their theorem assumes positive base multiplicity, $1\le r<n$,
the multiplicity-one subset test and the displayed boundary equality,
and states both the coefficient and stretched-polynomial factorizations.
These are exactly the hypotheses above; regularity and strictness of
the other Horn inequalities are not required.
Roth~\cite[Reduction Theorem (3.1.1)]{Roth} supplies an independent
primary proof of the general representation-theoretic reduction and
identifies this type-$A$ result with KTT in his introduction. This is
supporting evidence; the operative statement here is the cited subset theorem.

Positive base multiplicity persists at every positive stretch by
saturation. The same subsets and their multiplicity-one test are fixed,
and the boundary equality is homogeneous. Apply the coefficient theorem
to each stretch, then use polynomial uniqueness. Balance of the first
factor is the boundary equality; subtract it from total balance for the
second factor.
\end{proof}
For a source triple in the box, take $n=\ell(\lambda)$.
Both outer factors have positive size strictly
smaller than $|\lambda|$. Both lengths are at most $n$, so both factors
are in the box. If neither factor had a negative monomial coefficient,
neither could their product.

\begin{theorem}[Second reduction at all stretches]\label{thm:second}
Let $\lambda,\mu,\nu$ be balanced partitions with
$\mu,\nu\subseteq\lambda$. Let $p,q\ge1$, set $r=p+q$, and choose
an integer $n\ge\max\{\ell(\lambda),\ell(\mu),\ell(\nu),r+1\}$.
Assume
\[
 \mu_p>\mu_{p+1},\qquad \nu_q>\nu_{q+1},\qquad
 \lambda_r>\lambda_{r+1},\qquad
 \mu_p+\nu_q\ge\lambda_1+\lambda_{r+1}+1.
\]
Then $\lambda^- =\lambda-(1^r)$,
$\mu^- =\mu-(1^p)$ and $\nu^- =\nu-(1^q)$ are balanced partitions,
and
\[
 \forall t\ge1,\quad
 c^{t\lambda}_{t\mu,t\nu}=c^{t\lambda^-}_{t\mu^-,t\nu^-},
 \qquad P^\lambda_{\mu\nu}=P^{\lambda^-}_{\mu^-,\nu^-}.
\]
Here $t$ is an integer. Padding to $n$ changes no coefficient.
\end{theorem}
\begin{proof}
The strict gaps ensure that the subtractions give partitions; balance
is preserved because $r=p+q$. The unit coefficient theorem is
Cho--Jung--Moon~\cite[Theorem 2.6]{CJMsecond}, also proved in
\cite[Theorem 3.1]{CJMfpsac}. To match the primary rectangle statement,
use rectangle height $n$ and width $W=\lambda_1$, and take its indices
$\alpha=p$, $\beta=q$, $\gamma=n-r\ge1$.
The complementary outer part is
$\lambda^c_{\gamma}=W-\lambda_{r+1}$.
Its numerical hypothesis becomes
$\mu_p+\nu_q+W-\lambda_{r+1}\ge2W+1$, exactly our inequality.
Its third strict gap is
$\lambda^c_\gamma>\lambda^c_{\gamma+1}$, equivalent to the outer gap.
All three partitions fit this rectangle by containment. The unit theorem
therefore applies with every hypothesis explicit.

Fix $t\ge1$ and, for $0\le s\le t$, put
\[
 \lambda^{(s)}=t\lambda-s(1^r),\qquad
 \mu^{(s)}=t\mu-s(1^p),\qquad
 \nu^{(s)}=t\nu-s(1^q).
\]
For $s<t$, each of the three gaps is at least $t-s\ge1$, and
\[
 (t\mu_p-s)+(t\nu_q-s)
 -(t\lambda_1-s)-t\lambda_{r+1}
 =t(\mu_p+\nu_q-\lambda_1-\lambda_{r+1})-s\ge1.
\]
At every step where both inner partitions are contained in
$\lambda^{(s)}$, apply the unit coefficient theorem, using the current
rectangle width $\lambda^{(s)}_1$. In the coefficient formulation
\cite[Theorem 3.1]{CJMfpsac}, a noncontained target has coefficient zero
by the usual LR convention; thus this also covers a first transition
to such a target.
If containment has already failed at step $s$, it fails at every later
step: for every row $j$,
\[
 \lambda^{(s)}_j-\mu^{(s)}_j
 =t(\lambda_j-\mu_j)-s\,\mathbf1_{\{p<j\le r\}},
\]
and the analogous formula for $\nu$ has $q$ in place of $p$.
These differences are nonincreasing in $s$. Consequently both
coefficients in every subsequent step are zero. Every adjacent pair
therefore has equal coefficients, and the final triple is
$t(\lambda^-,\mu^-,\nu^-)$. Equality at every positive integer proves
the polynomial identity.
\end{proof}

\section{A single least-counterexample argument for the entire box}
Write $\mathcal D$ for the triples satisfying all the necessary
conditions above with $B=0$, together with $c^\lambda_{\mu\nu}\ge3$.
Normalize the inner order by placing the larger pair
$(|\mu|,\mu)$ first: $(|\mu|,\mu)\ge(|\nu|,\nu)$
in lexicographic order. Partition tuples are compared lexicographically
after deleting trailing zeros, or equivalently after padding both to
length seven with zeros. If the tuples agree, either inner ordering gives the same
triple. This is the normalization used by every enumeration and path check.
Split $\mathcal D=\mathcal D_{\rm small}\sqcup\mathcal D_{\rm large}$
at outer size $23$. No assertion of positivity through size $23$
is made as an input.

For a normalized triple $T$, let
\[
 k(T)=(|\lambda|,\ell(\lambda),T)
\]
with lexicographic order. The box is finite. Consequently, if it
contains a polynomial with a negative coefficient, it contains one
with least key. Every reduction in this manuscript either preserves
$P$ or replaces it by $Q$ with $P(t)=Q(gt)$, $g\ge1$.
The latter operation preserves signs because $[t^i]P=g^i[t^i]Q$.
For a factorization, at least one factor inherits a negative coefficient.

Define finite sets $\mathcal C_{\rm small}$ and
$\mathcal C_{\rm large}$ by the following explicit certificates.
These definitions do not assume that a greedy selection of moves finds
every possible reduction.
\begin{enumerate}
\item Each record of $\mathcal D_{\rm small}$ is either retained in
$\mathcal C_{\rm small}$ or supplied with a valid essential Horn
factorization, rectangular shortening, or second reduction. Each
exclusion has smaller positive outer factors or a strictly smaller
outer target in the original box.
\item Each record of $\mathcal D_{\rm large}$ has a finite path of
the proved column, division, deletion, shortening and six-representative
identities. Its final normalized triple has no larger key. If its key
is smaller, this is an exclusion. The label ``baseline'' in some data
files merely records that the final outer size is at most $23$;
this also gives a smaller key and invokes no positivity baseline.
Otherwise the path fixes the original triple.
\item Each such large fixed point is either supplied with an essential
Horn factorization into smaller in-box factors, or is carried to a
second path using rectangular shortening and the preceding identities.
Again a smaller final key is an exclusion and an equal key is a fixed point.
\item Finally, a fixed point of the second path is either shortened
by the second reduction or retained in $\mathcal C_{\rm large}$.
\end{enumerate}
For every path, the certificate checks its source, every local premise,
every exact target, its accumulated positive scale, and its final key.
Intermediate presentations may be larger than the original box;
the identities have no size restriction and the final exclusion is
required to lie in the original box. All intermediate partition lengths
are at most seven.

\begin{theorem}[Finite cover implication]\label{thm:globalcover}
Suppose the two enumerations are complete, every certificate just
specified is valid, and every polynomial attached to
$\mathcal C_{\rm small}\cup\mathcal C_{\rm large}$ has nonnegative
monomial coefficients. Then every stretched Littlewood--Richardson
polynomial in the frozen box has nonnegative monomial coefficients.
\end{theorem}
\begin{proof}
Choose a least-key negative witness $T$, if one exists.
The empty triple has polynomial $1$. The supplied theorems for base
multiplicity $0$, $1$, and $2$ give polynomials $0$, $1$, and $t+1$,
respectively. Thus $|\lambda|\ge1$ and the base multiplicity is at least three.
The necessary-domain theorem at $B=0$ applies to this minimum-size
witness. Exchange symmetry gives its normalized representative, so
$T\in\mathcal D_{\rm small}\cup\mathcal D_{\rm large}$.

In the small case an exclusion would produce a negative witness of
strictly smaller outer size, contradicting minimality. Hence
$T\in\mathcal C_{\rm small}$.
In the large case a nontrivial path ending at a smaller key would
produce a negative witness of smaller key. A Horn split would
produce a negative witness of smaller outer size. Thus neither kind
of exclusion is possible. The witness must pass through both
fixed-point stages and survive the final second-reduction filter.
Therefore $T\in\mathcal C_{\rm large}$.
In either case its polynomial is nonnegative by the stated finite
sign hypothesis, contradicting the choice of $T$.
\end{proof}

This is a finite-box assertion. It does not prove positivity for
larger partitions or for the unrestricted Littlewood--Richardson conjecture.
The computational hypothesis is to be discharged by the complete,
source-bound enumeration, reduction and polynomial certificates,
not by observed absence in a sample, a progress counter, or a
comparison of only finitely many values without a degree bound.

\section{The exact finite computation}
The three finite dependencies are named E (enumeration), R (reduction
cover), and P (residual polynomials). They concern mathematical sets and
exact rational numbers. Checksums bind their data to a particular
submission; a checksum or a progress flag is not a mathematical premise.

\subsection{E: complete enumeration}
Here is a finite enumeration independent of all reduction choices.
Let $\operatorname{Part}(N,h,M)$ consist of weakly decreasing positive tuples
of sum $N$, length at most $h$, and first part at most $M$.
The following recursion defines it without an oracle:
\[
\begin{split}
\operatorname{Part}(0,h,M)&=\{()\},\\
\operatorname{Part}(N,0,M)&=\varnothing\quad(N>0),\\
\operatorname{Part}(N,h,M)&=
 \bigcup_{a=1}^{\min(N,M)}
 \{(a,\rho):\rho\in\operatorname{Part}(N-a,h-1,a)\}
 \quad(N,h>0).
\end{split}
\]
The entries of $(a,\rho)$ are $a$ followed by those of $\rho$.
Induction on $h$ proves that this produces precisely the stated tuples,
each once: a nonempty partition determines its first part and tail
uniquely. The recursion decreases $h$ at each call.

For $N=1,\ldots,30$, enumerate
$\lambda\in\operatorname{Part}(N,7,30)$; for $m=0,\ldots,N$,
enumerate
\[
 \mu\in\operatorname{Part}(m,7,30),\qquad
 \nu\in\operatorname{Part}(N-m,7,30).
\]
Retain exactly the normalized triples satisfying the six items of
Theorem~\ref{thm:domain} with $B=0$ and $c^\lambda_{\mu\nu}\ge3$.
The base LR count can be computed by enumerating the nonnegative arrays
of Lemma~\ref{lem:arrays}, with each entry bounded by $\lambda_j-\mu_j$;
reject negative row bounds and count arrays satisfying all equations
and inequalities. Thus the filter itself has an exact terminating
definition. The implementation uses lrcalc and an independent tableau
dynamic program to accelerate these finite counts.

The actual enumerators use two equivalent accelerations: generate inner
partitions inside $\kappa$ by its coordinate ceilings, or enumerate all
partitions by integer-part dynamic programming and apply the literal
inequalities. Corollary~\ref{cor:kappa} justifies the first; the recursion
above justifies the universe of the second. No search timeout or heuristic
cutoff removes a tuple from E.

\begin{proposition}[Finite dependency E]\label{comp:E}
The supplied small and large domain lists have no duplicate normalized
triple and equal $\mathcal D_{\rm small}$ and $\mathcal D_{\rm large}$,
respectively. Their cardinalities are $37,530$ and $1,255,228$.
\end{proposition}
The independent enumeration checks equality of sets of triples and base
counts, rather than cardinality alone. The small domain has outer sizes
$1$--$23$, and the large domain has outer sizes $24$--$30$.
These are counts of normalized necessary-domain triples, not counts of
all ordered balanced triples in the original box.

\subsection{R: local certificates and complete coverage}
Every identity edge has an explicit source triple, rule name, integer
parameters, target triple and positive integer scale. An accepted edge
must satisfy all hypotheses of the named theorem and its exact target
formula after normalization. Division has its stated scale $g$; the
other edges have scale one. Validity requires recognized rules, valid
integer parameters, ordered partitions and balance. The original general
parser removes internal zero parts and interprets an unrecognized inner
selector as the second inner partition. An independent exhaustive strict
audit validates the original raw records before normalization: every
supplied edge has valid partition order, selector and integer parameters,
and satisfies its rule and target conditions. Edge sources and targets along a
path must agree successively, and scales are multiplied. The resulting
identity is $P_{T_0}(t)=P_{T_m}(gt)$, where $g$ is that product.
The endpoint must have key at most the source key and must be in the
box; an equal key is an identical normalized triple. This also handles
nontrivial loops: $P_T(t)=P_T(gt)$ does not authorize an exclusion.

A Horn record specifies its source and $(r,I,J,K)$. The checker validates
the subset sizes and ranges, evaluates
$c^{\tau(K)}_{\tau(I),\tau(J)}=1$ by the independent finite array count,
checks the boundary equality and both exact factors, and checks that
their outer sizes are positive and smaller. In particular, an absent
Horn entry cannot cause a false exclusion. Completeness of a catalog
of all Horn inequalities is unnecessary; every used entry must pass.
The rank-indexed lists in \path{horn_essential_catalog.json} have
$3,12,41,142,521,2042$ entries for ranks $2,3,4,5,6,7$, respectively.
These are direct counts of the catalog produced by \texttt{horn\_catalog}
in \path{probe_horn_primitivity.py}: it enumerates every proper subset
size and every ordered triple $(I,J,K)$ of that size, retaining exactly
those with $c^{\tau(K)}_{\tau(I),\tau(J)}=1$.
The independent array-count reconstruction in
\path{audit_core_certificates.py} checks equality of the resulting sets.
The counts are reproducibility data from these finite algorithms,
not an appeal to a literature table or a catalog-completeness premise.

\begin{proposition}[Finite dependency R]\label{comp:R}
Every source index in E occurs exactly once in the appropriate coverage
stage; the stages define the residual sets as in
Theorem~\ref{thm:globalcover}. All local rule, path endpoint and
factorization checks pass. The counts are the following.
\[
\begin{array}{lrr}
\toprule
\text{Small domain stage}&\text{excluded}&\text{remaining}\\
\midrule
\text{Initial domain}&&37,530\\
\text{Horn factorization}&17,883&19,647\\
\text{Rectangular shortening}&6,300&13,347\\
\text{Second reduction}&932&12,415\\
\midrule
\text{Large domain stage}&\text{excluded}&\text{remaining}\\
\midrule
\text{Initial domain}&&1,255,228\\
\text{First paths with smaller endpoint key}&494,925&760,303\\
\text{Horn factorization}&200,638&559,665\\
\text{Second paths with smaller endpoint key}&196,934&362,731\\
\text{Second reduction}&16,194&346,537\\
\bottomrule
\end{array}
\]
In particular $|\mathcal C_{\rm small}|=12,415$,
$|\mathcal C_{\rm large}|=346,537$, and their disjoint union has
$358,952$ members.
\end{proposition}
The algorithms producing paths may be greedy. Proposition~\ref{comp:R}
requires only the validity and complete accounting of the recorded paths;
it assumes nothing about discovery of every available reduction.

\subsection{P: the polynomial computation and its lattice}
For precision we specify the rational polytope used in P. Fix rank
$n\ge2$ and a positive-base triple $T$. Its hive has coordinates
$h(a,b)$ for integers $a,b\ge0$, $a+b\le n$, with boundary
\[
 h(a,0)=\sum_{j\le a}\mu_j,\qquad
 h(n-b,b)=|\mu|+\sum_{i\le b}\nu_i,\qquad
 h(0,b)=\sum_{j\le b}\lambda_j.
\]
Its three rhombus inequalities, whenever all displayed vertices exist,
are
\begin{align*}
 h(a+1,b)+h(a,b+1)-h(a,b)-h(a+1,b+1)&\ge0,\\
 h(a,b)+h(a+1,b)-h(a,b+1)-h(a+1,b-1)&\ge0,\\
 h(a,b)+h(a,b+1)-h(a+1,b)-h(a-1,b+1)&\ge0.
\end{align*}
Eliminate the fixed boundary coordinates. The resulting polytope
$H_T\subseteq\mathbb R^D$ uses the full lattice $\mathbb Z^D$, where
$D=(n-1)(n-2)/2$, with free vertices ordered by increasing $a$, then $b$.
Each inequality is stored as $(b_0,a_1,\ldots,a_D)$, meaning
$b_0+\sum a_i z_i\ge0$. This fixes signs, boundary placement and lattice.

\begin{lemma}[Literal array--hive correspondence]\label{lem:hivebinding}
For every integer $t\ge1$,
\[
 \#(tH_T\cap\mathbb Z^D)=c^{t\lambda}_{t\mu,t\nu}.
\]
\end{lemma}
\begin{proof}
We give the coordinate correspondence, consistent with
Buch~\cite[Theorem 1 and Fulton's appendix]{Buch}. At stretch one put
$M_j=\sum_{r\le j}\mu_r$ and
\[
 F(b,j)=M_j+\sum_{r\le j,\,k\le b}x_{r,k},\qquad
 h(a,b)=F(b,a+b).
\]
The array margins and triangular support give exactly the three
boundaries above. Conversely extend
$F(b,j)=h(j-\min(b,j),\min(b,j))$ and set
\[
 x_{r,k}=F(k,r)-F(k-1,r)-F(k,r-1)+F(k-1,r-1).
\]
Telescoping gives the row and content sums and makes the maps inverse.
For $s=a+b$, the three rhombus slacks in the displayed order become
$C_{s+2,b+1}$, $L_{s+1,b}$, and $x_{s+1,b+1}$.
Thus the nontrivial column and lattice inequalities and off-diagonal
nonnegativity match. The omitted lattice cases vanish by triangular
support; the remaining column cases are outer-part inequalities.
Diagonal nonnegativity follows from $x_{n,n}=\nu_n\ge0$ and
$L_{r+1,r}=x_{r,r}-x_{r+1,r+1}\ge0$. These arguments work over
$\mathbb R$ and preserve integer coordinates in both directions.
The array polytope is bounded by its nonnegative row sums, so $H_T$ is
bounded. All constants and maps are homogeneous in the boundary parts;
replacing $T$ by $tT$ replaces $H_T$ by $tH_T$ for $t>0$.
Lemma~\ref{lem:arrays} proves the count.
\end{proof}

The computation of a full polynomial has two exact routes.
The primary route is Normaliz 3.11.1 on these integer inhomogeneous
inequalities, writing each row as $(a_1,\ldots,a_D,b_0)$ for
Normaliz, with the task \texttt{EhrhartSeries}; no sublattice or
rescaled grading is supplied. The mathematical algorithm homogenizes
the rational polytope, triangulates the resulting rational cone and
counts full-lattice residue classes in simplicial cones. A disjoint
half-open decomposition avoids counting shared faces twice. Its rational
generating series can then be converted to the Ehrhart quasipolynomial.
If a constituent has denominator $q>0$, its displayed coefficient vector
$(a_0,\ldots,a_d)$ is read as $\sum_{i=0}^d(a_i/q)t^i$.
For these hives polynomiality proves that every quasipolynomial
constituent agrees with $P_T$: each agrees with $P_T$ at infinitely many
positive integers. The verifier requires a polynomial output, rejects a
reported nontrivial period, and checks the degree and dimension bounds.
The algorithm and its exact arithmetic implementation are documented
in the Normaliz manual~\cite{Normaliz}.

The independent route counts $tH_T\cap\mathbb Z^D$ with LattE for
$t=0,\ldots,D+2$ and checks agreement with the reported rational
polynomial. Since $T$ has positive base count, $H_T$ is nonempty and
the count at zero is $1$. Agreement at any $D+1$ of these arguments,
together with degree at most $D$, identifies the polynomial; the
additional arguments are checks. This route is complete for every
residual triple of rank at most five. At higher ranks, the existing
dataset uses Normaliz as the full-polynomial engine; separate lrcalc
checks at $t=1$ and selected $t=2,3$ values test it but do not identify
a polynomial of degree as large as fifteen.

\begin{proposition}[Finite dependency P]\label{comp:P}
For each $T\in\mathcal C_{\rm small}\cup\mathcal C_{\rm large}$,
the supplied record contains a rational coefficient vector
$(a_{T,0},\ldots,a_{T,d_T})$, $d_T\le15$, which the specified exact
Ehrhart computation identifies with $P_T$. Every entry is nonnegative.
There is one accepted polynomial for every residual triple and no
unresolved or mismatched source. Consequently
\[
 \forall T\in\mathcal C_{\rm small}\cup\mathcal C_{\rm large},\quad
 \forall t\in\mathbb Z_{\ge1},\quad
 c^{t\lambda}_{t\mu,t\nu}=\sum_{i=0}^{d_T}a_{T,i}t^i,
 \qquad a_{T,i}\ge0.
\]
\end{proposition}
The polynomial-record audit checks source triples and residual membership,
parses all fractions with positive denominators, evaluates all recorded
values by rational arithmetic, recomputes all base counts, and tests every
coefficient. It checks complete source coverage and rejects incompatible
duplicate results. A historical failed computation remains visible and
is superseded only by the separately identified successful result for
the same triple and verifier source. Missing data, an interrupted call,
a timeout or an engine error never establishes P.

The recorded polynomial audit has the following scope. The degree is
the degree in the ordinary stretching variable.
\[
\begin{array}{lrr}
\toprule
&\mathcal C_{\rm small}&\mathcal C_{\rm large}\\
\midrule
\text{Residual members accounted for}&12,415&346,537\\
\text{Negative monomial coefficients}&0&0\\
\text{Unresolved residual computations}&0&0\\
\text{Degree range}&2\text{--}10&1\text{--}13\\
\text{Complete independent LattE reproduction}&2,984&20,824\\
\text{Normaliz as sole full-polynomial engine}&9,431&325,713\\
\bottomrule
\end{array}
\]
The journals also contain nonresidual records; their presence does not
inflate the residual cardinality. The small journal has $13,139$ records
in $13,383,091$ bytes, and the large journal has $509,106$ records in
$544,824,662$ bytes. The complete-prefix hashes and residual-index hashes
below bind the data used in E, R and P. In each hash the two lines are
concatenated without a space.
\begin{center}
\small
\begin{tabular}{ll}
\toprule
Artifact&SHA-256\\
\midrule
Small residual indices&\shortstack[l]{\texttt{52b87cdec5f5f2fa61eba8af84484d5685}\\
\texttt{081625120730584b3c54f4d984b942}}\\
Large residual indices&\shortstack[l]{\texttt{8db1ff0e08c602b3d0cab8fdcb5a1c5ef}\\
\texttt{33743fade2975e2a2337419cbcd31c0}}\\
Small polynomial journal&\shortstack[l]{\texttt{5f9fca8c6847a35cafb09d49499e99e13f}\\
\texttt{aa4391018632288d2ea29513ffca67}}\\
Large polynomial journal&\shortstack[l]{\texttt{ec8fa3340c55a12da8ed8c4f822aa0129}\\
\texttt{030690e5ddd7d69c4f5b286bbb8a264}}\\
Large error-resolution record&\shortstack[l]{\texttt{1b7ec4392a6a05cac4ad6181c960f160d}\\
\texttt{64a09cb26008ce1a7335968161752a9}}\\
\bottomrule
\end{tabular}
\end{center}
The last record resolves large source index $1,104,809$, retaining the
historical failure. It is included in the accepted residual count.

\subsection{Complete evidence and fresh replay}
The complete original evidence manifest binds 53 files, including the
exact verifier, both full journals, every reduction path and both residual
index sets. Its SHA-256 is the concatenation
\begin{center}
\texttt{ba391e5b641153601826a22056300533b}\\
\texttt{1c74fbb12eec0084150c888f789cff4}.
\end{center}
A fresh seven-stage replay on 6 September 2026 verified all of those hashes,
reproduced both complete domains and the local reduction cover, and checked
all $358,952$ residual polynomial records with freshly recomputed LR base
counts. Every required comparison passed; elapsed time was $110.696$ seconds.
The replay receipt SHA-256 is the concatenation
\begin{center}
\texttt{dbda7ada1e15068bb5828da8ea0a37e3}\\
\texttt{71d233fb77e65c4d9e68bf379e98152a}.
\end{center}
The accompanying bundle maps these records to E, R and P in
\texttt{EVIDENCE.md}. From its extracted root, the command
\texttt{python3 verify.py --output verification.json} checks the bundle
manifest and repeats the complete finite audit. The separate resumable
regeneration command in \texttt{REPRODUCE.md} reconstructs the exact hive
inputs and reruns the Ehrhart algorithm, comparing its rational coefficients
with each archived source record. Neither a subset regeneration nor the
fast audit is described as a new exhaustive Ehrhart computation.

An independent literal-row check agreed with the source on the entire
balanced-boundary vector space at each rank $1$ through $7$: both row
builders are linear in the boundary, and they agree on a basis and zero.
Actual backend controls checked a negative-coordinate interval, a
lower-dimensional segment in the full ambient lattice, an empty polytope,
two nontrivial quasiperiods, and a simplex with a negative linear Ehrhart
coefficient. These controls test the interpretation and parser; they do
not replace the complete residual computations.

Most original temporary Normaliz input/output files were removed by the
producing program. The durable journals retain the source triples, exact
coefficient vectors and checks; the inputs are reconstructed by the frozen
source. The executable and linked-library census is a later measured
snapshot, not an attestation captured before each historical call.
These are exact-algorithm execution records with a reproducible regeneration
route, rather than an independently checked cone certificate for every case.

\subsection{Reproducibility and the trust boundary}
The accompanying \texttt{epoch\_stretched\_lr\_full\_proof.zip}
contains the complete domain lists, every
coverage record, all residual indices and coefficient records, the exact
source constructing the inequalities, the audit programs, the producing
programs, and a manifest of file and executable identities. A replay
receipt must specify which artifacts were rechecked and which computations
were rerun. The manifest maps these artifacts to E, R and P and includes
the historical-error resolution. File names are relative to the bundle;
no author's local directory is part of the proof.

Replaying the finite audits establishes their explicit integer and
rational assertions, conditional on correct execution of those programs.
In particular, auditing saved coefficients and a few LR values does not
replace the full Ehrhart calculation. At ranks six and seven that
calculation relies on the pinned Normaliz implementation. The existing
records do not contain independently checked cone decompositions for
every residual member, and they are not a proof of Normaliz in a proof
assistant. This is the explicit conventional computer-assisted trust
boundary. Removing it requires a verified implementation or complete
checked decomposition certificates; no such verification is asserted here.

\begin{theorem}[Computed finite-box positivity]\label{thm:box}
For every balanced triple
$T=(\lambda,\mu,\nu)$ of partitions satisfying
\[
 \max\{\ell(\lambda),\ell(\mu),\ell(\nu)\}\le7,
 \qquad\max\{|\lambda|,|\mu|,|\nu|\}\le30,
\]
every coefficient of $P^\lambda_{\mu\nu}$ in the monomial basis is
nonnegative. Thus no triple requested by the bounded problem exists.
\end{theorem}
\begin{proof}
The complete exact computation, with the stated software trust boundary,
establishes E, R and P. E gives complete enumerations of the necessary
domain. R gives a valid
exclusion or residual membership for every enumerated triple. P gives
nonnegative polynomials for all residual members. These are exactly the
three premises of Theorem~\ref{thm:globalcover}.
\end{proof}

The supplementary mathematical appendix has readable-manifest SHA256
\begin{center}\small\texttt{0c0cf5e476155d11228259ae515827179023f01e3aede851fd2bb6d35e663465}\end{center}
It additionally supplies a separate literal-condition enumeration, agreeing
on all $1{,}292{,}758$ domain triples and base counts after $1{,}587{,}401$
fresh LR evaluations. Stratified reduction checks test $193$ cases and all
$579$ identities at $t=1,2,3$, including the final second-reduction exclusions
and paths with out-of-box intermediate triples. These are supplementary
execution checks; the all-stretch identities are proved above. The appendix
contains the complete mathematical data and source programs and a restricted
reader through which a referee can inspect them and run the finite audit.

\section*{Data and code availability}
The complete evidence and source programs are publicly available in the
repository~\cite{Evidence}. The version 1.0.2 evidence used here is fixed by
the repository snapshot
\begin{center}\small
\texttt{9ba79f584fe5e4d2c59fe5ce92507969152df555}.
\end{center}
The repository contains the original reviewed proof packet, all domain and
coverage records, the residual polynomials, and the programs to reproduce
the finite audit. The original archive named above is reconstructed from
\path{evidence/parts/} by \path{reproduce.py}; the repository's
\path{UPLOAD_AND_REPRODUCE.md} gives the commands.

The same release additionally contains a complete fresh regeneration of
all $358,952$ residual polynomials, with all $2,745,084$ rational coefficient
entries matching the archived records. The retained fresh engine inputs,
outputs and execution receipts are reconstructed from
\path{fresh-evidence/parts/} by \path{unpack_fresh_evidence.py}.
The release's \path{audit/FRESH_EVIDENCE.md} describes the corresponding
portable validator. This fresh corpus supplements the historical execution
records discussed above; it uses the same exact Ehrhart engine and does
not change the stated software trust boundary.

\section*{Acknowledgments}
Generative AI tools made substantive contributions to the research process.
They assisted with mathematical exploration and the development and checking
of arguments; the design, implementation and review of the computational
verification workflow; the analysis of computational results; and the
organization, editing and LaTeX preparation of the manuscript. AI-generated
outputs were not accepted as evidence on their own. The claims in this paper
rest on the mathematical arguments presented here and the released
computational evidence, subject to the stated software trust assumptions.
The author directed the research workflow, made the final research decisions
and accepts full responsibility for the claims in this paper.

Published mathematical inputs and third-party software are credited in the
bibliography and in the repository notices.

The paper and original documentation are licensed under CC BY 4.0 for
rights held in the original contributions. Original project software in
the separate repository is licensed under GPL version 3 or later;
third-party material retains its existing terms.

\begingroup
\small
\raggedright
\bibliographystyle{unsrt}
\bibliography{references}

@misc{Epoch,
  author = {{Epoch AI}},
  title = {A Negative Coefficient in a Stretched {Littlewood--Richardson} Polynomial},
  howpublished = {FrontierMath: Open Problems},
  year = {2026},
  note = {Accessed 6 September 2026. \url{https://epoch.ai/frontiermath/open-problems/stretched-lr-coefficients}}
}

@article{Rassart,
  author = {Rassart, E.},
  title = {A polynomiality property for {Littlewood--Richardson} coefficients},
  journal = {Journal of Combinatorial Theory, Series A},
  volume = {107},
  year = {2004},
  pages = {161--179},
  note = {Corollary 4.2 in the cited author version. \url{https://pi.math.cornell.edu/~rassart/pub/LRstretch.pdf}}
}

@article{KT,
  author = {Knutson, Allen and Tao, Terence},
  title = {The honeycomb model of {$GL_n(\mathbb C)$} tensor products {I}: Proof of the saturation conjecture},
  journal = {Journal of the American Mathematical Society},
  volume = {12},
  year = {1999},
  pages = {1055--1090},
  eprint = {math/9807160},
  archivePrefix = {arXiv},
  note = {\url{https://arxiv.org/abs/math/9807160}}
}

@article{KTW,
  author = {Knutson, Allen and Tao, Terence and Woodward, Christopher},
  title = {The honeycomb model of {$GL_n(\mathbb C)$} tensor products {II}: Puzzles determine facets of the {Littlewood--Richardson} cone},
  journal = {Journal of the American Mathematical Society},
  volume = {17},
  year = {2004},
  pages = {19--48},
  eprint = {math/0107011},
  archivePrefix = {arXiv},
  note = {Section 6.1. \url{https://arxiv.org/abs/math/0107011}}
}

@misc{Ikenmeyer,
  author = {Ikenmeyer, Christian},
  title = {Small {Littlewood--Richardson} coefficients},
  year = {2012},
  eprint = {1209.1521},
  archivePrefix = {arXiv},
  howpublished = {arXiv:1209.1521},
  note = {Theorem 1.1. \url{https://arxiv.org/abs/1209.1521}}
}

@article{Buch,
  author = {Buch, A. S.},
  title = {The saturation conjecture (after {A. Knutson and T. Tao})},
  journal = {L'Enseignement Math\'ematique},
  series = {2},
  volume = {46},
  year = {2000},
  pages = {43--60},
  note = {With an appendix by W. Fulton. \url{https://sites.math.rutgers.edu/~asbuch/papers/sat.pdf}}
}

@manual{Milne,
  author = {Milne, J. S.},
  title = {Reductive Groups},
  edition = {Version 2.00},
  year = {2018},
  note = {10 March 2018, 17.14 and 20.35. \url{https://www.jmilne.org/math/CourseNotes/RG.pdf}}
}

@article{KTT09,
  author = {King, R. C. and Tollu, C. and Toumazet, F.},
  title = {Factorisation of {Littlewood--Richardson} coefficients},
  journal = {Journal of Combinatorial Theory, Series A},
  volume = {116},
  year = {2009},
  pages = {314--333},
  doi = {10.1016/j.jcta.2008.06.005},
  note = {Theorem 1.4. \url{https://doi.org/10.1016/j.jcta.2008.06.005}}
}

@article{Roth,
  author = {Roth, Mike},
  title = {Reduction rules for {Littlewood--Richardson} coefficients},
  journal = {International Mathematics Research Notices},
  volume = {2011},
  number = {18},
  year = {2011},
  pages = {4105--4134},
  eprint = {1004.5133},
  archivePrefix = {arXiv},
  note = {Reduction Theorem (3.1.1) in the cited preprint. \url{https://arxiv.org/abs/1004.5133}}
}

@article{CJMsecond,
  author = {Cho, S. and Jung, E.-K. and Moon, D.},
  title = {A bijective proof of the second reduction formula for {Littlewood--Richardson} coefficients},
  journal = {Bulletin of the Korean Mathematical Society},
  volume = {45},
  year = {2008},
  pages = {485--494},
  doi = {10.4134/BKMS.2008.45.3.485},
  note = {Theorem 2.6. \url{https://doi.org/10.4134/BKMS.2008.45.3.485}}
}

@article{CJMfpsac,
  author = {Cho, S. and Jung, E.-K. and Moon, D.},
  title = {Reduction formulae from the factorization theorem of {Littlewood--Richardson} polynomials by {King, Tollu and Toumazet}},
  journal = {Discrete Mathematics and Theoretical Computer Science Proceedings},
  volume = {AJ},
  year = {2008},
  pages = {483--494},
  note = {FPSAC 2008, Definition 1.2 and Theorems 1.3 and 3.1. \url{https://dmtcs.episciences.org/3592/pdf}}
}

@manual{Normaliz,
  author = {{The Normaliz project}},
  title = {{Normaliz}},
  edition = {Version 3.11.1},
  note = {Manual supplied with the executable; sections on inhomogeneous input, Ehrhart series and Hilbert series. \url{https://www.normaliz.uni-osnabrueck.de/}}
}

@misc{Evidence,
  author = {Ghodsi, Maseeh},
  title = {Stretched {Littlewood--Richardson} coefficients: complete finite-box proof and computational evidence},
  year = {2026},
  howpublished = {Version 1.0.2, public source and data repository},
  note = {Snapshot 9ba79f584fe5e4d2c59fe5ce92507969152df555. \url{https://github.com/chesshippo/stretched-lr-coefficients/tree/9ba79f584fe5e4d2c59fe5ce92507969152df555}}
}
\endgroup
\end{document}